\documentclass{amsart}

\usepackage{a4wide,amssymb}

\let\ge\geqslant

\def\1{^{-1}}

\def\CP{{\mathbf C\mathbf P}}

\newtheorem{theorem}{Theorem}[section]
\newtheorem{lemma}[theorem]{Lemma}
\newtheorem{proposition}[theorem]{Proposition}

\theoremstyle{definition}

\theoremstyle{remark}

\numberwithin{equation}{section}

\begin{document}

\title[]{Computing The Krichever genus}

\begin{abstract} Let $\psi$ denote the genus
that corresponds to the formal group law having
invariant differential $\omega(t)$ equal to
$\sqrt{1+p_1t+p_2t^2+p_3t^3+p_4t^4}$
and let
$\kappa$ classify the formal group law strictly isomorphic to the universal formal group law under strict isomorphism $x\CP(x)$.
We prove that on the rational complex bordism ring the Krichever-H\"ohn genus $\phi_{KH}$ is the composition $\psi\circ \kappa^{-1}$.
We construct certain elements $A_{ij}$ in the Lazard ring and give an alternative definition of the universal Krichever formal group law. We conclude that
the coefficient ring of the universal Krichever formal group law is the quotient of the Lazard ring by the ideal generated by all $A_{ij}$, $i,j\geq3$.
\end{abstract}

\author{Malkhaz Bakuradze }
\address{Faculty of exact and natural sciences, Iv. Javakhishvili Tbilisi State University, Georgia }
\email{malkhaz.bakuradze@tsu.ge}
\thanks{The author was supported by Volkswagen Foundation, Ref.: I/84 328 and by Rustaveli NSF, DI/16/5-103/12 }

\maketitle

\section{Rational Krichever-H\"ohn genus}

For the current state of complex cobordism and formal group laws we refer the reader to the excellent survey \cite{BUCH2}.

A formal group law over a commutative ring with unit $R$ is a power series $F(x,y)\in R[[x,y]]$ satisfying

\medskip

(i) $F(x,0)=F(0,x)=x$,

(ii) $F(x,y)=F(y,x)$,

(iii) $F(x,F(y,z))=F(F(x,y),z)$.

\medskip

Let $F$ and $G$ be formal group laws. A homomorphism from $F$ to $G$ is a power series $\nu(x)\in R[[x]]$ with constant term $0$ such that
$$\nu(F(x,y))=G(\nu(x),\nu(y)).$$

It is an isomorphism if $\nu'(0)$ (the coefficient at $x$) is a unit in $R$, and a strict isomorphism if the coefficient at $x$ is 1.

If $F$ is a formal group law over a commutative $\mathbb{Q}$-algebra $R$, then it is strictly isomorphic to the additive formal group law $x+y$. In other words, there is a strict isomorphism $l(x)$ from $F$ to the additive formal group law, called the logarithm of $F$, so that $F(x,y)=l^{-1}(l(x)+l(y))$.  The inverse to logarithm is called the exponential of $F$.

The logarithm $l(x)\in R \otimes \mathbb{Q}[[x]]$ of a formal group law $F$ is given by
$$
l(x)=\int_{0}^{x}\frac{dt}{\omega(t)},\,\,\,\,\,\omega(x)=\frac{\partial F(x,y)}{\partial y}(x,0).
$$

There is a ring $L$, called the universal Lazard ring, and a universal formal group law $F(x,y)=\sum a_{ij}x^iy^j$ defined over $L$. This means that for any formal group law $G$ over any commutative ring with unit $R$ there is a unique ring homomorphism $r:L\rightarrow R$ such that $G(x,y)=\sum r(a_{ij})x^iy^j$.

The formal group law of geometric cobordism was introduced in \cite{NOV}. Following Quillen we will identify it with the universal Lazard formal group law as it is proved in \cite{Q} that the coefficient ring of complex cobordism $MU_*=\mathbb{Z}[x_1,x_2,...]$, $|x_i|=2i$ is naturally isomorphic as a graded ring to the universal Lazard ring.

\bigskip

The Krichever-H\"{o}hn genus or the general four variable complex elliptic genus \cite{hoehndiss},\cite{KR},
$$\phi_{KH}: MU_{*}\otimes \mathbb{Q}\to \mathbb{Q}[q_1,...,q_4]$$
is a graded $\mathbb{Q}$-algebra homomorphism defined by the following property: if one denotes by $f_{Kr}(x)$ the exponential of the Krichever \cite{BUCH-BU} universal formal group law $\mathcal{F}_{Kr}$, then the series
$$
h(x):=\frac{f_{Kr}'(x)}{f_{Kr}(x)}
$$
satisfies the differential equation
\begin{equation}
\label{eq:condition-h}
(h'(x))^2=S(h(x)),
\end{equation}
where $S(x)=x^4+q_1x^3+q_2x^2+q_3x+q_4$, the generic monic polynomial of degree 4 with formal parameters
$q_i$ of weights $|q_i|=2i$.

\medskip

This genus is the universal genus on the rational bordism ring of $SU$-manifolds which is multiplicative in fiber bundles of
$SU$-manifolds with compact connected structure group.

\bigskip

To generalize the Ochanine  elliptic genus from $\Omega^{SO}_*\otimes \mathbb{Q}$ to $\mathbb{Q}[\mu,\epsilon]$ (see \cite{OSH}),
 a new elliptic genus $\psi$ is defined in \cite{SCH},
 $$\psi: MU_{*}\otimes \mathbb{Q}\to \mathbb{Q}[p_1,p_2,p_3,p_4],$$
 to be the genus whose logarithm equals

$$
\int_{0}^{x}\frac{dt}{\omega(t)} ,\,\,\,\,\omega(t)=\sqrt{1+p_1t+p_2t^2+p_3t^3+p_4t^4}
$$
and $p_i$ are again formal parameters $|p_i|=2i$.

\medskip

It is proved in \cite{SCH} that the $\psi$-genus is the universal genus on the rational complex bordism ring
which is multiplicative in projectivizations $P(E)$ of complex vector bundles $E \to B$ over
Calabi-Yau 3-folds $B$ (i.e. $B$ is a compact K\"{a}hler manifold with vanishing first Chern class).

Clearly to calculate the values of $\psi$ on $\CP_i$, the generators of the rational complex bordism ring $MU_{*}\otimes\mathbb{Q}=\mathbb{Q}[\CP_1,\CP_2,\dots ]$ we need only the Taylor expansion of $(1+y)^{-1/2}$ as by above definition
$$(1+p_1x+p_2x^2+p_3x^3+p_4x^4)^{-1/2}=log'_{\psi}=\sum_{i\geq 1}\psi(\CP_i)x^i.$$

A straightforward calculation shows that $\psi$ is surjective \cite{SCH}:

\medskip

$\psi(\CP_1)=-\frac{1}{2}p_1$;

$\psi(\CP_2)=\frac{3}{8}p_1^2-\frac{1}{2}p_2$;

$\psi(\CP_3)=-\frac{5}{16}p_1^3+\frac{3}{4}p_1p_2-\frac{1}{2}p_3$;

$\psi(\CP_4)=\frac{35}{128}p_1^4-\frac{15}{16}p_1^2 p_2+\frac{3}{8}p_2^2+\frac{3}{4}p_1p_3-\frac{1}{2}p_4$.

\medskip

It is natural to ask whether we can similarly calculate $\phi_{KH}$
in an elementary manner,
different from that relying on the formulas in \cite{BUCH-PANOV-RAY} and \cite{hoehndiss}.

\medskip

Let $\kappa$ be the classifying map of the formal group law $\tilde{F}$ over the rational Lazard ring defined as follows.

Let $\CP(x)=1+\sum_{i\geq 1}\CP_i x^i$, where $\CP_i$ is the bordism class of the complex projective space $CP^i$ and let
$$\nu(x):=x\CP(x)$$
be the strict isomorphism $F\to \tilde{F}$, where $F$ is the universal formal group law, so that
$$\nu(F(x,y))=\tilde{F}(\nu(x),\nu(y)).$$
Now let
$$\kappa:MU_{*} \otimes \mathbb{Q}\to MU_{*}\otimes \mathbb{Q}$$
be the classifying map of $\tilde{F}$ and let $log(x)$ and $mog(x)$ be the logarithm series of $F$ and $\tilde{F}$ respectively then by definition
$$
mog(x)=log(\nu^{-1}(x)), \,\,\,\,\,\, \kappa(log(x))=mog(x).
$$

Therefore the value $\kappa(\CP_i)$ is determined by equating the coefficients at $x^i$ in

\begin{equation}
\label{eq:kappa}
\sum_{i\geq 1}\frac{\kappa(\CP_i)}{i+1}x^{i+1}=\sum_{i\geq 1}\frac{\CP_i}{i+1}(\nu^{-1}(x))^{i+1}.
\end{equation}

For instance

\medskip

$\kappa(\CP_1)=-\CP_1$;

$\kappa(\CP_2)=3\CP_1^2-2\CP_2$;

$\kappa(\CP_3)=-10\CP_1^3+12\CP_1\CP_2-3\CP_3$;

$\kappa(\CP_4)=35\CP_1^4-60\CP_1^2\CP_2+20\CP_1\CP_3+10\CP_2^2-4\CP_4$.

\bigskip

\bigskip

The following theorem shows how $\phi_{KH}$ is related to $\psi$.

\begin{theorem}
\label{formula} Let $F$ be the universal formal group law and $\tilde{F}$
its strictly isomorphic formal group law under strict isomorphism $\nu(x)=x\CP(x)$.
Let $t:\mathbb{Q}[p_1,...,p_4]\to \mathbb{Q}[q_1,...,q_4]$ be the ring isomorphism defined by $t(p_i)=q_i$ and
$\kappa:MU_{*} \otimes \mathbb{Q}\to MU_{*}\otimes \mathbb{Q}$ the classifying map of $\tilde{F}$. Then

\medskip

i) the pair $ (t,\,\,\sum_{i\geq 0}\phi_{KH}(\CP_i)x^{i+1})$
is the strict isomorphism from  $\phi_{KH}(F)$ to $\psi (F)$, i.e.,
the series $\sum_{i\geq 0}\phi_{KH}(\CP_i)x^{i+1}$ is the strict isomorphism  from $\phi_{KH}(F)$ to $t\psi(F)$ in the usual sense.

\medskip

ii) A method to compute $\phi_{KH}$ is given by the formula
$\phi_{KH}=t\circ\psi\circ \kappa^{-1}$.
\end{theorem}

\medskip

To establish the theorem, we need the following two lemmas.

\begin{lemma}
\label{lem1}
Let $exp$ be the exponent of $F$. The series $\frac1{h(x)}=exp(x)/exp'(x)$ is invertible.
Furthermore, the inverse
 $j(x)$ of this series coincides with then $mog(x)$, the logarithmic series of $\tilde{F}$.
\end{lemma}

\begin{proof} Note that
$$
j^{-1}(log(x))=\frac{exp(log(x))}{exp{'}(log(x))}=x\CP(x).
$$
Hence by definition of $\nu(x)$ one has

$$
j^{-1}(mog(x))=j^{-1}(log(\nu^{-1}(x)))=\nu^{-1}(x)\CP(\nu^{-1}(x))=x,
$$
as $x\CP(x)=\nu(x)$.
\end{proof}

\medskip

\begin{lemma}
\label{KH=F} Let $\tilde{\omega}=\frac1{mog'(x)}$ be the invariant differential form of the formal group law $\tilde{F}$ above.
Then the condition of Krichever-H\"ohn \eqref{eq:condition-h} is satisfied if and only if $\tilde{\omega}(x)^2$ is a polynomial of degree 4.
\end{lemma}

\begin{proof}
Since the series $j(x)$ is invertible, the condition \eqref{eq:condition-h} is equivalent to
$$
(h'(j(x)))^2=S(h(j(x))).
$$
But $h(j(x))=1/x$, hence
$$
(h\circ j)'(x)=h'(j(x))j'(x)=-1/x^2,
$$
so that by Lemma \ref{lem1}
$$
h'(j(x))=-\tilde{\omega}(x)/x^2.
$$
It follows that the condition \eqref{eq:condition-h} is equivalent to
$$
\tilde{\omega}(x)^2=x^4S(1/x),
$$
where on the right hand side one clearly has a degree four polynomial.
\end{proof}

 Now let us prove Theorem \ref{formula}. Lemma \ref{KH=F} implies that $\phi_{KH}(\tilde{F}))$ is of type $\psi(F)$, that is  corresponding invariant form $\omega(x)^2$ is a polynomial of degree 4. By definition the formal group law $\psi(F)$ is universal with this property. Therefore there is classifying map of $\phi_{KH}(\tilde{F})$, that is unique ring homomorphism
 $t:\mathbb{Q}[p_1,...,p_4]\to \mathbb{Q}[q_1,...,q_4]$ such that $t(\psi(F))=\phi_{KH}(\tilde{F})=\phi_{KH} (\kappa (F))$. Therefore $t\circ \psi=\phi_{KH}\circ \kappa $. This proves i). For ii) note that by definition $\kappa$ is isomorphism and we get
 $t\circ \psi \circ \kappa ^{-1}=\phi_{KH}$. Finally $t(p_i)=q_i$ as $t$ is unique and it sends ${1+p_1t+p_2t^2+p_3t^3+p_4t^4}$ to ${1+q_1t+q_2t^2+q_3t^3+p_4t^4}$.
 \qed

\bigskip

\section{Integral Krichever genus}

Now we turn to the universal Krichever formal group law $F_{Kr}$ \cite{KR} and prove that it coincides with the universal formal group law by Buchstaber $F_{B}$ (with a minor specialisation that does not affect the formal group law).

In \cite{BUCH1} V. M. Buchstaber has given the analytical solution of a functional equation for the exponent of the formal group law of the form

\begin{equation}
\label{eq:FB1}
\mathcal{F}_{B}(x,y)=\sum \alpha_{ij}x^iy^j=\frac{A(y)x^2-A(x)y^2}{B(y)x-B(x)y}.
\end{equation}

\medskip

Note that if our series $A$ and $B$ have the form
$$A(t)=A_0+A_1t+A_2t^2+O(t^3),$$
$$B(t)=B_0+B_1t+B_2t^2+O(t^3),$$
then the coefficient $B_1=B'(0)$ does not affect the formal group law.

\bigskip

\begin{lemma}
\label{eq:B(t)}
Let $\omega(x)=\frac{\partial F(x,y)}{\partial y}(x,0)$ be the invariant form of the universal formal group law $F$, and
let $\mathcal{F}$ be the formal group law of the form \eqref{eq:FB1}, with $B'(0)=A'(0)$. Then the invariant form of $\mathcal{F}$ equals $B(x)$, i.e., $B(x)$ is the image of $\omega(x)$ under the ring homomorphism classifying the formal group law $\mathcal{F}$.
\end{lemma}
\begin{proof}
To see this note that

$$
\mathcal{F}(x,y)=\frac{A_0}{B_0}(x+y)+O(xy).
$$
So if $\mathcal{F}(x,y)$ is a formal group law we must have $A_0=B_0$, and after dividing the numerator and denominator appropriately we may assume that $A_0=B_0=1$. We then furthermore calculate
$$
\mathcal{F}(x,y)=x+y+A_1xy+\sum_{i=2}^\infty B_i(x^iy+xy^i)+O(x^2y^2).
$$

We thus have
$$\alpha_{1i}=B_i, \,i\geq 2, \,\, \alpha_{11}=A_1=A'(0)=B'(0)$$
and
$$
1+\sum_{i\ge1}{\alpha}_{1i}t^i
$$
is restricted $(1+\sum_{i\ge1}[\CP^i]t^i)^{-1}=\omega(t)$.
\end{proof}

Let us now present a minor modification of the analysis of $\mathcal{F}_{B}$ performed in \cite{roindiss} and as Lemma \ref{eq:B(t)} suggests to introduce
\begin{equation}
\label{eq:Aij}
A(x,y)=\sum A_{ij}x^iy^j=F(x,y)(x\omega(y)-y\omega(x)).
\end{equation}

 We define the universal Nadiradze formal group law $\mathcal{F}_{N}$ by the obvious classifying map of the Lazard ring to its  quotient ring by the ideal generated by all $A_{ij}$ with $i,j\geq 3$.

\bigskip

\begin{proposition}
\label{eq:A(x,y)}
Let $L$ denote the Lazard ring.

\medskip

i) In $L[[x]]$, the identity
$\omega'(x) - \omega'(0) = 2 x \hat{\omega}(x)$ holds, where $\hat{\omega}(x) = \sum_{i \geqslant 1} \omega_i x^{i-1}$.

ii) The formal series $A(x,y)$ satisfies  the identity
\begin{equation*}
A(x,y)=\left(x\omega(y)+y\omega(x)-\omega'(0)xy\right)\left(x\omega(y)-y\omega(x)\right)+\left(\omega(x)\hat{\omega}(x)-
\omega(y)\hat{\omega}(y)\right)x^2y^2
\end{equation*}
in $L[[x,y]]/(xy)^3$.
\end{proposition}

\begin{proof}
Let $f$ and $g$ be the exponent and logarithm of $F$, respectively.
Hence $F(x,y)=f(g(x)+g(y))$ and
\begin{equation}
\label{eq:f'}
f'(x)=1/g'(f(x))=\omega(f(x)), \,\,\,f'(g(x))=\omega(f(g(x))=\omega(x).
\end{equation}
Let $\omega(x)=1+b_1x+b_2x^2+\cdots $.
Since $g''(0)=-f''(0)=-\omega'(0)=-b_1$, we conclude
\begin{equation}
\label{eq:f''gx}
\frac{\partial^2 F }{\partial y^2}(x,0)=f^{''}(g(x))+f'(g(x))g''(0)=\omega'(x)\omega(x)-\omega'(0)\omega(x).
\end{equation}
This implies i), since the left hand side of \eqref{eq:f''gx} has factor 2, and $\omega(x)$ is invertible. ii) Because of antisymmetry we have modulo $(xy)^3$
$$
A(x,y)=A(y)x^2-A(x)y^2=\sum (A_{i2}x^2y^i-A_{i2}x^iy^2).
$$
We want to calculate $-\sum A_{i2}x^i$ in terms of $\omega(x)$.

\medskip

Applying $\frac{\partial^2}{\partial y^2}(x,0)$ to \eqref{eq:Aij} and taking into account \eqref{eq:f'} and \eqref{eq:f''gx} we obtain
$$
-2\sum A_{i2}x^i=x\omega(x)\omega'(x)-x\omega'(0)\omega(x)+2x\omega'(0)\omega(x)-2\omega^2(x)+2b_2x^2.
$$
Since the coefficients are in the Lazard ring, this reasoning implies
$$
-\sum A_{i2}x^i=x\omega(x)\frac{\omega'(x)-\omega'(0)}{2}+x\omega'(0)\omega(x)-\omega^2(x)+b_2x^2.
$$
Consequently
\begin{align*}
\sum (A_{i2}x^2y^i-A_{i2}x^iy^2)&=
(x\omega(y)+y\omega(x))(x\omega(y)-y\omega(x))-\omega'(0)xy(x\omega(y)-y\omega(x))\\
&\quad
+\omega(x)\hat{\omega}(x)x^2y^2-\omega(y)\hat{\omega}(y)x^2y^2.
\end{align*}
\end{proof}

In order to compute the Krichever genus on the coefficients of the formal group law of geometric cobordism, in \cite{BUCH-BU}, the universal Krichever formal group law $\mathcal{F}_{Kr}$ is defined as

\begin{equation}
\label{eq:F_N}
\mathcal{F}_{Kr}(x,y)=xb(y)+yb(x)-b'(0)xy+
\frac{b(x)\beta(x)-b(y)\beta(y)}{xb(y)-yb(x)}x^2y^2,
\end{equation}
where $\beta(x)=\frac{b'(x)-b'(0)}{2x}$.
In \cite{BUCH-BU}, it is moreover proved that
$$b(x)=\frac{\partial \mathcal{F}_{Kr}}{\partial y}(x,0).$$
Lemma \ref{eq:B(t)} and Proposition \ref{eq:A(x,y)} ii) imply that
$\mathcal{F}_{Kr}$ can alternatively be defined by the classifying map of
$\mathcal{F}_N$ if $B(x)$ is written as $b(x)$. Thus $\mathcal{F}_{N}=\mathcal{F}_{Kr}$.

Following \cite{BUCH1},  the authors of \cite{BUCH-PANOV-RAY} consider the following formal group law corresponding to the Krichever genus
$$
\mathcal{F}_{b}(u_1,u_2)=u_1c(u_2)+u_2c(u_1)-au_1u_2-\frac{d(u_1)-d(u_2)}{u_1c(u_2)-u_2c(u_1)}u_1^2u_2^2.
$$
It follows from Lemma \ref{eq:B(t)} and Proposition \ref{eq:A(x,y)} ii)
that, if we take $c'(0)=a$,
the Krichever genus
$\mathcal{F}_{b}(x,y)$ coincides with \eqref{eq:F_N}, that is $c(x)=b(x)$ and $d(x)=-b(x)\beta(x)$.

\bigskip

Thus we get the following

\begin{theorem} Let $F$ be the universal formal group law, let
$\omega(x)=\frac{\partial F(x,y)}{\partial y}(x,0)$ be its invariant form, and let
$$
\sum A_{ij}x^iy^j=F(x,y)(x\omega(y)-y\omega(x)).
$$
The Buchstaber, Krichever and Nadiradze  formal group laws coincide, that is,
$$
\mathcal{F}_{b}=\mathcal{F}_{Kr}=\mathcal{F}_{N},
$$
and the coefficient ring is the quotient of the Lazard ring  by the ideal generated by all $A_{ij}$ with $i,j\geq 3$.
\end{theorem}

\bigskip

In \cite{B-J}, we calculated the coefficient ring of the Nadiradze formal group law $\mathcal{F}_{N}$ up to dimension 26. Namely, there is a set
of polynomial generators $z_1,z_2,...$ of the Lazard ring for which the low degree defining relations are
$$
5z_5=z_2z_3+2z_1z_4,\,\,2z_6=0,\,\,z_1z_6=0,\,\,z_3z_6=0,\,\,z_{10}=0,\,\,z_5z_6=0,\,\,z_{12}=0,
$$
and $7z_7,$ $2z_8,$ $3z_9,$ $11z_{11},$ and $13z_{13}$ are decomposable.
For reasons of space, we omit here these long decompositions.
We note that our calculations agree with the results in \cite{BUCH-BU} on the structure of the coefficient ring of $\mathcal{F}_{Kr}$ obtained in terms of
the associativity equation. The new information here concerning the Krichever group and hence the Krichever genus is that,
in dimensions 20 and 24, there are no indecomposable elements,
because in these dimensions
$z_{10}=0$ and $z_{12}=0$.

The question arises in which dimensions any element is multiplicatively decomposable.

\end{document}